\documentclass[12pt,twoside, psamsfonts]{amsart}
\usepackage{amsfonts}
\usepackage{amsmath, amsthm, amssymb, mathrsfs, paralist,tabularx,supertabular, verbatim, amsthm, manfnt}
\usepackage[all]{xy}

\newcommand{\Z}{{\mathbb Z}}

\newcommand{\twod}{{\operatorname{2-dense}}}

\evensidemargin \oddsidemargin
\parskip=\medskipamount

\newtheorem{thm}{Theorem}[section]
\newtheorem{theorem}[thm]{Theorem}
\newtheorem{corollary}[thm]{Corollary}
\newtheorem{lemma}[thm]{Lemma}	
\newtheorem{proposition}[thm]{Proposition}

\theoremstyle{definition}
\newtheorem{definition}[thm]{Definition}

\theoremstyle{remark}

\newtheorem*{lemma*}{Lemma}

\numberwithin{equation}{section}

\title{Polynomials with divisors of every degree}

\author{Lola Thompson}
\address{Department of Mathematics\\ 
6188 Kemeny Hall\\
Dartmouth College\\
Hanover, NH 03755, USA}
\email[] {Lola.Thompson@Dartmouth.edu}

\begin{document}

\begin{abstract} 
\noindent We consider polynomials of the form $t^n-1$ and determine when members of this family have a divisor of every degree in $\Z[t]$. With $F(x)$ defined to be the number of such integers $n \leq x$, we prove the existence of two positive constants $c_1$ and $c_2$ such that $$c_1 \frac{x}{\log x} \leq F(x) \leq c_2 \frac{x}{\log x}.$$ \end{abstract}

\maketitle 

\section{Introduction and statement of results}

Which polynomials with integer coefficients have integral divisors of every degree? The trivial answer is that $f(t) = 0$ is the unique polynomial with this property. However, if we clarify the problem by specifying that we are interested in polynomials $f(t)$ with divisors of every degree up to $\deg f(t)$, then the question becomes more interesting. Certainly, any polynomial that splits completely into linear factors, such as $f(t) = t^n$, satisfies this criterion. However, there are other choices of polynomials that are not as obvious. In this paper, we examine polynomials of the form $t^n-1$, where $n$ is a positive integer. 

In order to determine the values of $n$ for which $t^n-1$ has a divisor of every degree up to $n$, it will be helpful to use the following identity: \begin{equation}\label{cyclotomic} t^n-1 = \prod_{d \mid n} \Phi_d(t),\end{equation} where $\Phi_d(t)$ is the $d^{th}$ cyclotomic polynomial. Since $\deg \Phi_d(t) = \varphi(d)$ and each $\Phi_d(t)$ is irreducible, then the following statements are equivalent:

\indent \textbf{(1)} The polynomial $t^n-1$ has a divisor of every degree between $1$ and $n$.
 
\indent \textbf{(2)} Every integer $m$ with $1 \leq m \leq n$ can be written in the form $$m = \sum_{d \in \mathcal{D}} \varphi(d),$$ where $\mathcal{D}$ is a subset of divisors of $n$. 

We will call such a positive integer $n$ \textit{$\varphi$-practical}. The nomenclature stems from the striking similarity between the statement in \textbf{(2)} and the definition of a practical number. A positive integer $n$ is \textit{practical} if every $m$ with $1 \leq m \leq n$ can be written as a sum of distinct positive divisors of $n$; that is, $m = \sum_{d \in \mathcal{D}} d$, where $\mathcal{D}$ is a subset of the divisors of $n$.

In this paper, we prove the following results on $\varphi$-practical numbers:

\begin{theorem}\label{density} The set of $\varphi$-practical numbers has asymptotic density $0$. \end{theorem}

\begin{theorem}\label{count} Let $F(x) = \# \{n \leq x : n$ is $\varphi$-practical$\}$. There exist two positive constants $c_1$ and $c_2$ such that for $x \geq 2$, we have $$c_1 \frac{x}{\log x} \leq F(x) \leq c_2 \frac{x}{\log x}.$$ \end{theorem}

While Theorem \ref{count} immediately implies Theorem \ref{density}, there is a much simpler proof of Theorem \ref{density} that we will present in section 2. 

In order to prove Theorem \ref{count}, we will rely on several results and tools developed in the literature on practical numbers, which we will now outline in a brief history. The term ``practical'' was coined by Srinivasan in 1948.  In 1950, Erd\H{o}s \cite{erdos} remarked, without proof, that the practical numbers have asymptotic density 0. In 1954, B. M. Stewart \cite{stewart} gave a classification of the practical numbers: \begin{proposition}[Stewart's Condition] If $n = p_1^{e_1} p_2^{e_2} \cdots {p_j}^{e_j}$, where $p_1 < p_2 < \cdots < p_j$ are primes and $e_i \geq 1$ for $i = 1, \cdots, j$, then $n$ is practical if and only if for every $i$, $p_i \leq \sigma(p_1^{e_1} p_2^{e_2} \cdots p_{i-1}^{e_{i-1}}) + 1$, where $\sigma$ is the sum-of-divisors function. (Note that Stewart's Condition implies that all practical numbers except $n=1$ are even.)\end{proposition} 

Let $P\!R(x) = \#\{n \leq x : n \ \hbox{is practical}\}.$ Determining the true size of $P\!R(x)$ has been of interest for some time. In 1986, Hausman and Shapiro \cite{h&s} showed that there exists a positive constant $C_\beta$ such that $$P\!R(x) \leq C_\beta \frac{x}{(\log x)^\beta}$$ for every fixed $\beta < 2^{-1}(1/\log 2 - 1)^2 = 0.0979.$ This result was improved upon by Tenenbaum \cite{tenenbaum} in the same year, who showed that for $\lambda = 4.20002$ and for $x \geq 16$, $$\frac{x}{\log x}(\log \log x)^{- \lambda} \ll P\!R(x) \ll \frac{x}{\log x} \log \log x \log \log \log x.$$ Based on computational data, Margenstern \cite{marg} conjectured in 1991 that $P\!R(x) \sim c x/\log x$, where $c$ is a positive constant. This conjecture was partially proven in 1997 by Saias \cite{saias}, who showed that there exist two strictly positive constants $c_3$ and $c_4$ such that for $x \geq 2$, we have 

\begin{equation}\label{saiaseqn} c_3 \frac{x}{\log x} \leq P\!R(x) \leq c_4 \frac{x}{\log x}.\end{equation}

\noindent We use this theorem and its proof in the proof of Theorem \ref{count}.

It is interesting to note that our work on the $\varphi$-practical integers allows us to classify a second family of polynomials with divisors of every degree. Namely, $t^n+1$ has an integral divisor of every degree up to $n$ if and only if $n$ is odd and $\varphi$-practical. This follows from the fact that, when $n$ is even, $t^n+1$ has no divisor of degree $1$. On the other hand, when $n$ is odd, we have $t^n+1 = -((-t)^n-1)$, hence $t^n+1$ has divisors of all of the same degrees as those of $t^n-1$.

Throughout this paper, we will use the following notation. Let $n$ be a positive integer. Let $\omega(n)$ denote the number of distinct prime factors of $n$ and let $\Omega(n)$ denote the number of prime factors of $n$ counting multiplicity. We will use $\tau(n)$ to designate the number of positive divisors of $n$. Let $P(n)$ denote the largest prime factor of $n$, with $P(1) = 1$, and let $\Psi(x, y) = \#\{n \leq x : P(n) \leq y \}$. Moreover, let $P^-(n)$ denote the smallest prime factor of $n$, with $P^-(1) = + \infty$. Lastly, we will use $\log_k(x)$ to denote the $k^{th}$ iterate of the natural logarithm function.

%%%Proof of Theorem 1.2
\section{Proof of Theorem \ref{density}}

Below we present our proof of Theorem \ref{density}, which we believe is likely to be similar to the argument that Erd\H{o}s had in mind for the practical numbers.

\begin{proof}  From the definitions of the functions $\omega(n), \tau(n), \Omega(n)$, it is clear that $2^{\omega(n)} \leq \tau(n) \leq 2^{\Omega(n)}$. Fix $\varepsilon = 1/1000$. Since $\omega(n)$ and $\Omega(n)$ both have normal order $\log \log n$ (cf.\cite[Theorem 431]{h&w}), then for all $n$ except for a set with asymptotic density 0, we have

\begin{equation} \label{logineq} 2^{(1-\varepsilon) \log \log n} \leq \tau(n) \leq 2^{(1+ \varepsilon) \log \log n} = (\log n)^{(1+ \varepsilon) \log 2} < (\log n)^{0.7}. \end{equation} 

We can factor the polynomial $t^n-1 = \prod_{d \mid n} \Phi_d(t)$, where $\Phi_d(t)$ is the $d^{th}$ cyclotomic polynomial. Since each $\Phi_d(t)$ is irreducible, the number of divisors of $t^n -1$ in $\Z[t]$ is $2^{\tau(n)}$, since every divisor is uniquely determined by deciding whether or not to include each $\Phi_d(t)$ with $d \mid n$ in its factorization. Thus, in order for $n$ to be $\varphi$-practical, we need $n \leq 2^{\tau(n)}$; otherwise, $t^n-1$ would not have a divisor of every degree less than or equal to $n$. Taking the logarithm of both sides of this inequality and combining it with \eqref{logineq}, we have $$\log n \leq \tau(n) \log 2 < \tau(n) < (\log n)^{0.7}.$$ But this is impossible, so the numbers $n$ that are $\varphi$-practical are in the set with asymptotic density 0 where \eqref{logineq} does not hold.\end{proof}

%%%Upper Bound Proof
\section{Proof of the upper bound of Theorem \ref{count}}

Stewart's Condition shows the form that every practical number must take. The key to proving this is a recursive argument showing that each practical number $M$ can be used to generate new practical numbers via the following set of conditions:

\begin{lemma}[Stewart]  \label{stewartkey}If $M$ is a practical number and $p$ is a prime with $(p, M) = 1$, then $M' = p^k M$ is practical (for $k \geq 1$) if and only if $p \leq \sigma(M) + 1$. \end{lemma}

Stewart's Condition would be a simple corollary of Lemma \ref{stewartkey} if it were not for the following subtlety: while Lemma \ref{stewartkey} provides a method for building an infinite family of practical numbers, it is not immediately obvious that all practical numbers arise in the prescribed manner. Stewart's Condition confirms our suspicions.

The simple necessary-and-sufficient condition in Lemma \ref{stewartkey} turns out to be a powerful tool. In addition to being an important component in the proof of Stewart's condition, it is also used in Saias' proofs of the upper and lower bounds for the size of $P\!R(x)$. Unfortunately, we have not found such a simple statement for the $\varphi$-practical numbers. Stewart's Condition implies that each practical number $M' > 1$ can be constructed by multiplying a smaller practical number $M$ by a prime power $p^k$, where $p > P(m)$. However, the same cannot be said for the $\varphi$-practical numbers. For example, $315 = 3^2 \cdot 5 \cdot 7$ is $\varphi$-practical, but $45 = 3^2 \cdot 5$ is not, since there are no totient-sum representations for $22$ and $23.$ 

A more natural means of classifying the $\varphi$-practical numbers would be to use the following criterion: Let $w_1 \leq w_2 \leq \cdots \leq w_k$ be the set of totients of divisors of a positive integer $n$, rearranged so that they appear in non-decreasing order. Then $n$ is $\varphi$-practical if and only if, for each $i < k$, we have $$w_{i+1} \leq 1 + w_1 + \cdots + w_i.$$ Unfortunately, this criterion for $\varphi$-practicality is not particularly useful to us, since the totients of divisors of $n$ are not monotonic in general.  

To get around these problems, we will only give a necessary condition for a number to be $\varphi$-practical, which is all that is needed in order to determine the stated upper bound for $F(x)$. In section 4, we will give a necessary-and-sufficient condition for a squarefree integer to belong to the set of $\varphi$-practical numbers, which will be used to obtain the lower bound for $F(x)$ in section 5.

\begin{definition} Let $n = p_1^{e_1} \cdots p_k^{e_k}$, where $p_1 < p_2 < \cdots < p_k$ are primes and $e_i \geq 1$ for $1 \leq i \leq k$. Define $m_i = p_1^{e_1} \cdots p_i^{e_i}$ for $i = 0,..., k-1.$ We say that such an integer $n$  is \textit{weakly $\varphi$-practical} if the inequality $p_{i+1} \leq m_i + 2$ holds for $i=0,...,k-1$. \end{definition}

\begin{lemma}\label{necessary} Every $\varphi$-practical number is weakly $\varphi$-practical. \end{lemma}

\begin{proof} Let $n = p_1^{e_1} \cdots p_k^{e_k}$, with $p_1 < p_2 < \cdots < p_k$ and $e_i \geq 1$ for $i \leq 1 \leq k$. Suppose that there exists an integer $i$ for which $p_{i+1} > m_i + 2$. Observe that, if $i = 0$, then $m_0 = 1$. Hence, if $p_{i+1} > m_i + 2$ holds at $i = 0$, we must have $p_1 > 3.$ Then, $n > 3$ and $t^n-1$ has no divisor of degree $2$, so $n$ is not $\varphi$-practical. Thus, we may assume that $i > 0$. Now, $p_{i+1} > m_i + 2$ implies that $\varphi(p_{i+1}) > m_i + 1$. Moreover, it is always the case that $m_i = \sum_{d \mid m_i} \varphi(d)$. Hence, if $d \mid n$ and $d \nmid m_i$, then $\varphi(d) > m_i + 1.$ In particular, $t^n - 1$ has no divisor of degree $m_i + 1$. Therefore, $n$ cannot be $\varphi$-practical.\end{proof}

The converse to Lemma \ref{necessary} is false. For example, $45$ is not $\varphi$-practical, but it is weakly $\varphi$-practical. We can use Lemma \ref{necessary} in order to obtain the stated upper bound for $F(x)$. 

\begin{lemma}\label{pproduct} If $n$ is practical and $p \leq P(n)$, then $pn$ is practical. The same holds for weakly $\varphi$-practical numbers. \end{lemma}

\begin{proof} This is immediate from Stewart's condition and from the definition of weakly $\varphi$-practical numbers. \end{proof}

\begin{lemma}\label{evenpractical} Every even weakly $\varphi$-practical number is practical. \end{lemma}

\begin{proof} Let $n$ be an even weakly $\varphi$-practical number with $\omega(n) = k$. Since $n$ is weakly $\varphi$-practical, it must be the case that $p_{i+1} \leq m_i + 2$ for all $i < k$. Furthermore, since $n \geq 2$, we have $m_i + 2 \leq \sigma(m_i) + 1$ for all $i \geq 1$. Hence, each $p_{i+1}$ satisfies the inequality from Lemma \ref{stewartkey}, so $n$ is practical.\end{proof} 

\begin{theorem} There exists a positive constant $c_2$ such that, for $x \geq 2$, we have $$F(x) \leq c_2 \frac{x}{\log x}.$$ \end{theorem}

\begin{proof} If $n$ is a $\varphi$-practical number then, by Lemma \ref{necessary}, $n$ is weakly $\varphi$-practical. Thus, if $n$ is even, Lemma \ref{evenpractical} implies that $n$ is practical. If $n$ is odd, then $2^{\ell}n$ is practical for every $\ell \geq 1$, by Lemmas \ref{pproduct} and \ref{evenpractical}. Moreover, for each odd integer $n$ in $(0, x]$, there is a unique positive integer $\ell_0$ such that $2^{\ell_0}n$ is in the interval $(x, 2x].$ Therefore, we have \begin{align*} F(x) &= \#\{n \leq x : n \ \hbox{even and $\varphi$-practical}\} + \#\{n \leq x : n \ \hbox{odd and $\varphi$-practical}\}\\ & \leq \#\{n \leq x : n \ \hbox{is practical}\} + \#\{x < m \leq 2x: m \ \hbox{is practical}\} \\& = P\!R(2x).\end{align*} By \eqref{saiaseqn}, we have $$P\!R(x) \leq c_4 \frac{x}{\log x}.$$ Taking $c_2 = 2 c_4$, we obtain $F(x) \leq P\!R(2x) \leq c_2 \frac{x}{\log x}.$ \end{proof} 

%%%Key lemmas for the lower bound
\section{Preliminary lemmas for the lower bound of Theorem \ref{count}}

In order to acquire the stated lower bound for the size of the set of $\varphi$-practical numbers, it suffices to find a lower bound for the size of the set of squarefree $\varphi$-practicals. Although we were unable to give a necessary-and-sufficient condition that characterizes all $\varphi$-practical numbers, we are able to find such a condition for the squarefree $\varphi$-practical numbers. This condition will play a crucial role in section 5, when we give a proof for the lower bound of $F(x)$. In order to obtain this condition, we will need the following lemma, which is our analogue to Lemma \ref{stewartkey}.

\begin{lemma}\label{keylemma} If $M$ is $\varphi$-practical and $p$ is prime with $(p, M) = 1$, then $M' = pM$ is $\varphi$-practical if and only if $p \leq M + 2$. Moreover, $M' = p^{k}M, k \geq 2$ is $\varphi$-practical if and only if $p \leq M + 1$. \end{lemma}

\begin{proof} For the first case, we take $M' = pM$. If $p > M+2$, then Lemma \ref{necessary} implies that $M'$ cannot be $\varphi$-practical.

For the other direction, we assume that $p \leq M + 2$ and $M' = pM$. Suppose that we can write an integer $n$ in the form $n = (p-1) q + r$, with $0 \leq q, r \leq M$. Since $q, r \leq M$ and $M$ is $\varphi$-practical, we can write $q = \sum_{d \in \mathcal{D}} \varphi(d), r = \sum_{d' \in \mathcal{D}'} \varphi(d')$, for some subsets $\mathcal{D}, \mathcal{D}'$ of divisors of $M.$ Then

$$n = \displaystyle\sum_{pd \in p\mathcal{D}} \varphi(pd) + \displaystyle\sum_{D \in \mathcal{D}'} \varphi(D)$$ where $p\mathcal{D} = \{pd : d \in \mathcal{D}\}.$ There is no overlap between $p\mathcal{D}$ and $\mathcal{D'}$, since the first set only contains divisors of $pM$ that are not divisors of $M$. So, there exists a polynomial with degree $n$ that divides $t^{pM} - 1$.

Thus, in order to conclude that $M'$ is $\varphi$-practical, it remains for us to show that every integer $n \leq M'$ can be written in the form $(p-1) q + r$, with $0 \leq q, r \leq M$. We will break $[0, M']$ into subintervals of the form $[(p-1)q, (p-1)q + M]$. Since $p \leq M +2$ then $(p-1)q + M \geq (p-1)q + (p-2),$ which is adjacent to $(p-1)(q+1).$ Thus, all of the intervals are overlapping or, at least, contiguous. Moreover, the first subinterval starts at $0$ and the last subinterval ends at $M'.$ Thus, $M'$ is $\varphi$-practical.

For the second case, we take $M' = p^k M, k \geq 2$. We have seen that $p \leq M + 2$. Now, suppose that $p = M + 2$. Then, from the first case, we know that $pM$ is $\varphi$-practical. However, the smallest irreducible divisor of $x^{M'} -1$ that has degree larger than $pM$ has degree $\varphi(p^2)$. Since $p = M+2$, we have $$\varphi(p^2) = M^2 + 3M + 2 > M^2 + 2M  + 1 = pM + 1,$$ so there is no divisor of $t^{M'} - 1$ with degree $pM + 1$. Thus, $M'$ is not $\varphi$-practical if $p = M + 2$. 

For the other direction, we assume that $p \leq M+1$. We will use induction on the power of $p$, taking the case where $M' = pM$ to be our base case. For our induction hypothesis, we assume that $p^{k-1} M$ is $\varphi$-practical. Now, suppose that $n \in [0, p^k M]$. Let $q_1$ be the largest integer in $[0, M]$ with $\varphi(p^k) q_1 \leq n$. If $q_1 = M,$ then $$n - \varphi(p^k) q_1 = n - \varphi(p^k) M \leq (p^k - \varphi(p^k)) M = p^{k-1} M.$$ By our induction hypothesis, $p^{k-1} M$ is $\varphi$-practical, so we have $$n - \varphi(p^k)M = \sum_{d \in \mathcal{D}} \varphi(d)$$ where $\mathcal{D}$ is a subset of divisors of $p^{k - 1}M$. Thus, we can write $$n = \sum_{d \in \mathcal{D}} \varphi(d) + \sum_{d \mid M} \varphi(p^k d).$$ Therefore, when $q_1 = M$, we see that $n$ is $\varphi$-practical. If $q_1 < M$ then, using the assumption that $p \leq M + 1$, we have $$n - \varphi(p^k)q_1 < \varphi(p^k)(q_1 + 1) - \varphi(p^k)q_1 = \varphi(p^k) = p^{k-1}(p - 1) \leq p^{k-1}M.$$ Once again, we see that our induction hypothesis implies that $n$ is $\varphi$-practical.\end{proof}

Recall the definition of weakly $\varphi$-practical from section $3.$

\begin{corollary}\label{weakly} A squarefree integer $n$ is $\varphi$-practical if and only if it is weakly $\varphi$-practical. \end{corollary}

\begin{proof} Let $n = p_1 \cdots p_k$, with $p_1 < p_2 < \cdots < p_k$, and suppose that $n$ is weakly $\varphi$-practical. We proceed by induction on the number of prime factors of $n$. For our base case, we observe that $n = 1$ is both weakly $\varphi$-practical and $\varphi$-practical. Suppose that all squarefree integers $n$ with at most $k-1$ prime factors that are weakly $\varphi$-practical are, in fact, $\varphi$-practical. Hence, since $\frac{n}{p_k}$ is weakly $\varphi$-practical, it is also $\varphi$-practical, according to our induction hypothesis. But then $n = p_k \cdot \frac{n}{p_k}$ with $\frac{n}{p_k}$ $\varphi$-practical, and $p_k \leq \frac{n}{p_k} + 2$, since $n$ is weakly $\varphi$-practical. Therefore, by Lemma \ref{keylemma}, $n$ is $\varphi$-practical. The other direction of the proof is an immediate consequence of Lemma \ref{necessary}.\end{proof}

%%%Lower Bound Proof
\section{Proof of the lower bound of Theorem \ref{count}}

Throughout the remainder of this paper, we will use the following notation. Let $$F'(x) = \# \{n \leq x : n \ \mathrm{is} \ \varphi \hbox{-practical} \ \mathrm{and} \ \mathrm{squarefree} \}.$$ 

Let $1 = d_1(n) < d_2(n) < \cdots  < d_{\tau(n)}(n) = n$ denote the increasing sequence of divisors of $n.$ Let $p_1 < p_2 < \cdots < p_{\omega(n)}$ be the increasing sequence of prime factors of $n$. For integers $n$, we define $$T(n) = \mathrm{max}_{1 \leq i < \tau(n)} \frac{d_{i+1}(n)}{d_i(n)}.$$ Let $$D(x, y, z) = \#\{1 \leq n \leq x : T(n) \leq z, P(n) \leq y \ \mathrm{and} \ n \ \mathrm{is} \ \mathrm{squarefree} \},$$ and let $$D(x) = D(x, x, 2).$$ \begin{definition} An integer $n$ is called $z$-\textit{dense} if $n$ is squarefree and $T(n) \leq z.$ \end{definition} Note: This is not the way that Saias defines $2$-dense integers. His definition does not include the stipulation that $n$ is squarefree. For the purposes of this paper, we will only need to consider squarefree integers. The results of Saias that we cite below are valid for squarefree $n$. 

Using the notation defined above, we see that $D(x, x, z)$ counts the number of $z$-dense integers up to $x$. Saias notes that the set of $2$-dense integers is properly contained within the set of squarefree practical numbers. As a result, if $P\!R'(x)$ denotes the number of squarefree practical numbers up to $x$, then a lower bound for $D(x)$ will also be a lower bound for $P\!R'(x).$ The bulk of Saias' work is, therefore, in obtaining a lower bound for $D(x, y, z)$, where $2 \leq z \leq y \leq x$. In the particular case when $x = y$ and $z = 2$, he obtains the following inequalities, the first of which immediately yields his stated lower bound for $PR(x)$: 

\begin{lemma}[Saias]\label{SaiasD} There exist positive constants $\kappa_1$ and $\kappa_2$ such that $$\kappa_1 \frac{x}{\log x} \leq D(x) \leq \kappa_2 \frac{x}{\log x}$$ for all $x \geq 2$. \end{lemma}

Unfortunately, the same relationship does not exist between $F'(x)$ and $D(x)$. For example, $66$ is $2$-dense but not $\varphi$-practical. In order to get around this problem, we introduce the following modified definition of $2$-dense integers:

\begin{definition}\label{std} A squarefree integer $n$ is \textit{strictly $2$-dense} if $\frac{d_{i+1}}{d_i} < 2$ holds for all $i$ satisfying $1 < i < \tau(n) - 1,$ and $\frac{d_2}{d_1} = 2 = \frac{d_{\tau(n)}}{d_{\tau(n) - 1}}$. Note that this forces $n$ to be even. \end{definition} Although this modification is subtle, it is sufficient for removing the non-$\varphi$-practical $2$-dense integers from our consideration. 

\begin{lemma}\label{modified} Every strictly $2$-dense number is $\varphi$-practical. \end{lemma}

\begin{proof} Write $n = p_1 p_2 \cdots p_k$, where $2 = p_1 < p_2 < \cdots < p_k.$ If $k = 1$, then the only strictly $2$-dense integer with exactly $1$ prime factor is $n = 2$, which is also $\varphi$-practical. Assume that $k > 1$ and $n$ is not $\varphi$-practical. Then, as $n$ is squarefree, Lemma \ref{weakly} implies that $n$ is not weakly $\varphi$-practical either. As a result, the inequality from the definition of weakly $\varphi$-practical numbers must fail for some prime $p_j$ dividing $n$, with $j > 1$. Let $n_j = \prod_{i < j} p_i$, so $p_j > n_j + 2$. Now, the largest proper divisor of $n_j$ is at most $\frac{n_j}{2}.$ Moreover, there are no divisors of $n$ between $\frac{n_j}{2}$ and $n_j$, since all of the other prime factors of $n$ are greater than or equal to $p_j$. Therefore, there exist divisors $d_i, d_{i+1}$ of $n$ with $\frac{d_{i+1}}{d_i} \geq 2$, namely $d_i = \frac{n_j}{2}$ and $d_{i+1} = n_j.$ Since $j > 1$, then $d_{i+1} = n_j > 2 = d_2$, so $i > 1.$ On the other hand, since $n_j < p_j$, we must have $i < \tau(n) - 1$. Thus, we have shown that there exists an index $i$ with $1 < i < \tau(n) - 1$ such that $\frac{d_{i+1}}{d_i} \geq 2$, so $n$ is not strictly $2$-dense.  \end{proof}

Let $$D'(x) = \#\{1 \leq n \leq x : n \ \mathrm{is} \ \mathrm{strictly} \ 2\hbox{-dense} \}.$$ From Lemma \ref{modified}, a lower bound for $D'(x)$ will also serve as a lower bound for $F'(x)$. One might wonder why we have only defined $D'(x)$ in terms of a single parameter, while Saias' function $D(x)$ is initially defined in terms of both $x$ and $y$. This departure stems from a difference in our approaches. Saias' proofs for $D(x)$ relied on a number of iterative arguments that restricted the size of the largest prime factor of $n$, but our proofs for $D'(x)$ will not require any similar restrictions. 

In his proof of the lower bound for $D(x, y)$, Saias relies heavily on the following condition for $2$-dense numbers:

\begin{lemma}[Tenenbaum] \label{TCondition} For every integer $n \geq 1$, $T(n) \leq 2$ if and only if $$T(n/P(n)) \leq 2$$ and $$P(n) \leq \sqrt{2n}.$$ \end{lemma} 

\noindent Like Saias, our proof will rely heavily on an analogue of Lemma \ref{TCondition} for strictly $2$-dense integers, which we will give below. Although the statements of these lemmas are quite similar, their proofs differ substantially. The proof of Lemma \ref{TCondition} relies on a structure theorem of Tenenbaum \cite[Lemma 2.2]{tenenbaum} that describes all $z$-dense numbers in terms of their prime factorization, while our approach will not make use of any heavy machinery. We will prove our version of Lemma \ref{TCondition} in three parts.

\begin{lemma}\label{cases} For every squarefree integer $n > 1$, $n$ is strictly $2$-dense if $n/P(n)$ is strictly $2$-dense and $P(n) < \sqrt{n}$. \end{lemma}

\begin{proof} Assume that $n$ is squarefree, $m = n/P(n)$ is strictly $2$-dense, and $P(n) < \sqrt{n}$. In other words we are assuming that $P(n)^2 < n,$ so $P(n) < \frac{n}{P(n)} = m$. Suppose that the divisors of $m$ are $1 = d_1 < d_2 < \cdots < d_k = m$. Let $P(n) = p.$ Then $p = p d_1 < p d_2 < \cdots < p d_k = pm$, along with the divisors of $m$, form the divisors of $pm.$ Now, since $m$ is strictly $2$-dense, it follows that $m$ is even, hence $n$ is even as well. As a result, we have $$\frac{d_2}{d_1} = 2 = \frac{pd_k}{pd_{k-1}}.$$ Thus, the only ratios that may pose an obstruction to $mp$ being strictly $2$-dense are $\frac{d_k}{d_{k-1}}$ and $\frac{p d_2}{p d_1}$. In order to show that these ratios do not cause a problem, we will show that there exist divisors $d_i, d_j$ of $m$ with $p d_i \in (\frac{m}{2}, m)$ and $d_j \in (p, 2p).$ The general principle behind the argument is that, if $m$ is strictly $2$-dense and $x$ is a real number with $1 < x < m/2$, then $m$ has a divisor in the interval $(x, 2x),$ since otherwise we would have two consecutive divisors $d_i, d_{i+1}$ with $1 < i < \tau(n) -1$ and $d_i \leq x$, $d_{i+1} \geq 2x$, i.e. $\frac{d_{i+1}}{d_i} \geq 2.$ There are three cases to consider:

\textit{Case 1:} If $p > \frac{m}{2}$ then, since $p < m$, we have $p < m < 2p$. Moreover, since $\frac{m}{2} < p < m$, we see that both conditions are satisfied; that is, we can let $d_i = 1$ and $d_j = m$.

\textit{Case 2:} If $\frac{m}{4} < p < \frac{m}{2}$, then $p < \frac{m}{2} < 2p < m$. Hence, we have $p < \frac{m}{2} = d_{k-1} < 2p$ and $\frac{m}{2} < 2p = p d_2 < m$, so both conditions are met.

\textit{Case 3:} If $p < \frac{m}{4}$, consider the interval $(p, 2p)$. Since $p < \frac{m}{4}$, then $2p < \frac{m}{2} = d_{k-1},$ hence the existence of a divisor $d_j$ of $m$ with $d_j \in (p, 2p)$ follows from the fact that $m$ is strictly $2$-dense. Furthermore, since $m$ is strictly $2$-dense, there exists a divisor $d_i$ of $m$ in the interval $(\frac{m}{2p}, \frac{m}{p})$. Therefore, $p d_i \in (\frac{m}{2}, m).$\end{proof}

\begin{lemma}\label{pinduct} If a squarefree integer $m$ satisfies Definition \ref{std} for all ratios of divisors $\frac{d_{i+1}}{d_i}$ up to $d_{i+1} = P(m)$, then $m$ is strictly $2$-dense. \end{lemma}

\begin{proof}

We proceed by induction on the number of distinct prime factors of $m$. For our base case, we observe that if $m$ has $1$ distinct prime factor, then the only integer $m$ for which Definition \ref{std} holds for all divisors up to $P(m)$ is $2,$ which is strictly $2$-dense. For our induction hypothesis, we suppose that if $m$ has at most $\ell - 1$ distinct prime factors and $\frac{d_{i+1}}{d_i} < 2$ holds for all pairs of consecutive divisors $(d_i, d_{i+1})$ up to $d_{i+1} = p_{\ell-1}$, then $m$ is strictly $2$-dense. Now, consider an integer $m$ with $\ell$ distinct prime factors, i.e., $m = p_1 \cdots p_\ell$, and suppose that the strictly $2$-dense definition holds for all consecutive pairs of divisors $(d_i, d_{i+1})$ up to $d_{i+1} = p_\ell$. If $p_\ell > \frac{m}{p_\ell}$, then there is no divisor of $m$ between $\frac{m}{2 p_\ell}$ and $\frac{m}{p_\ell}$. This contradicts our assumption that all pairs of consecutive divisors $(d_i, d_{i+1})$ up to $d_{i+1} = p_\ell$ satisfy $\frac{d_{i+1}}{d_i} < 2.$ Thus, it must be the case that $p_\ell < \frac{m}{p_\ell}$, i.e., $p_\ell < \sqrt{m}.$ Moreover, since $P(\frac{m}{p_\ell}) = p_{\ell-1} < p_\ell$, then our induction hypothesis implies that $\frac{m}{p_\ell}$ is strictly $2$-dense. Therefore, $m$ is strictly $2$-dense by Lemma \ref{cases}.\end{proof} 

%%%Strictly 2-dense Condition
\begin{lemma} \label{T'Condition} For every squarefree integer $n > 6$, $n$ is strictly $2$-dense if and only if $n/P(n)$ is strictly $2$-dense and $P(n) < \sqrt{n}.$ \end{lemma}

\begin{proof} Let $n$ be a squarefree integer larger than $6$, let $m = n/P(n)$, and let $p = P(n).$ Assume that $n$ is strictly $2$-dense. Then, $n$ must be even, hence $m$ is even. First, we assume that $p > m$. Then, we have $\frac{m}{2} < m < p$ as consecutive divisors of $n$. Since $n > 6$ then $m = \frac{n}{p} > \frac{6}{3} = 2$, hence $\frac{m}{2} > 1$. Therefore, there exists some integer $1 < i < \tau(n)-1$ for which $d_i = \frac{m}{2}$, $d_{i+1} = m$ and $d_{i+2} = p$. Clearly, $\frac{d_{i+1}}{d_i} \geq 2$, so $n$ cannot be strictly $2$-dense. 

Secondly, continuing with our assumption that $n$ is strictly $2$-dense, we suppose that $m$ is not. Then, there is some pair of divisors $d_i, d_{i+1}$ of $m$, with $1 < i < k-1$, for which $\frac{d_{i+1}}{d_i} \geq 2$. Without loss of generality, we may assume that $d_i$, $d_{i+1}$ is the smallest pair with this property. Now, in order for $n$ to be strictly $2$-dense, it must be the case that $p$ falls between $d_i$ and $d_{i+1},$ since $p$ is the smallest divisor of $n$ that is not also a divisor of $m$. However, since $p$ is the largest prime divisor of $n$, then $P(m) < p$, so the strictly $2$-dense definition holds for all divisors of $m$ up to $P(m)$. But, if the strictly $2$-dense definition holds for all divisors of $m$ up to $P(m)$ then, by Lemma \ref{pinduct}, $m$ must be strictly $2$-dense.  However, this contradicts our prior assumption that $m$ is not strictly $2$-dense. Therefore, if $m$ is not strictly $2$-dense then $n$ cannot be strictly $2$-dense.

The second direction of the proof follows immediately from Lemma \ref{cases}.\end{proof}                                                                                                                                                                                                                                                                                                                                                                                                                                                                                                                                                                                                                                                                                                                                                                                                                                                                                                                                                                                                                                                                                                                                                                                                                                                                                                                                                                                                                                                                                                                                                                                                                                                                                                                                                                     

In order to obtain a lower bound for $D'(x)$, we will use Lemmas \ref{TCondition} and \ref{T'Condition} to show essentially that a positive proportion of $2$-dense integers are strictly $2$-dense. Thus, Saias' lower bound for $D(x)$ will also serve as a lower bound for $D'(x)$. Before we commence with the proof, we will pause to discuss a technique from sieve methods that will be useful in this context. For the remainder of this section, we will use the following notation. Let $u = \frac{\log x}{\log y}.$ Let $\rho(t)$ be the Dickman function, i.e., $\rho(t)$ is the continuous real-valued function with $\rho(t) = 1$ for $t \leq 1$ and for $t > 1$, $\rho(t)$ satisfies the differential equation $t\rho'(t) + \rho(t - 1) = 0.$ As in \cite{saias}, we will define $$H(x, y) = \left\{
\begin{array}{l l}   
    \frac{x \log 2}{\log x}(1 - \frac{1}{\log_2(\log x/\log 2)}) \rho(u(1- \frac{1}{\sqrt{\log y}}) - 1) & (0 < u < 3(\log x)^{1/3}),\\
    \Psi(x, y) & (u \geq 3(\log x)^{1/3}).
\end{array}\right.$$ Saias constructs $H(x, y)$ to serve as a more tractable model for $D(x, y)$. After much difficult work, he shows that $$D(x) \asymp H(x, x) \asymp \frac{x}{\log x}.$$ In addition, he proves that $H(x,y)$ satisfies the following Buchstab-type inequality:
\begin{lemma}[Saias]\label{saiasmain} For $x \geq 2^{16}, y \geq 2$ and $0 < u < 3(\log x)^{1/3}$, we have $$H(x, y) \geq 1 + \sum_{p \leq \mathrm{min}(y, \sqrt{2x}\ell(x))} H(x/p, p),$$ where $\ell(x) = \exp \{\frac{\log x}{(\log_2 x)(\log_3 x)^3}\}.$\end{lemma} 

\noindent In other words, $H(x, y)$ was constructed so that it has roughly the same order of magnitude as $D(x, y)$ and, like $D(x, y)$, it can be expressed recursively using a Buchstab inequality. In fact, Saias was able to demonstrate a more explicit relationship between these two functions, which will be useful in our lower bound argument for $D'(x)$:

\begin{lemma}[Saias]\label{newsaias} There exists a constant $c \geq 1$ such that, under the conditions $x \geq 2$ and $y \geq 2$, we have $$D(x, y) \leq c H(x, y).$$ \end{lemma}

Now that we have some information about the behavior of the function $H(x, y)$ and its relationship with $D(x,y)$, we have all of the ingredients necessary to prove our lower bound for $D'(x)$. 

%%% Main Lower Bound Argument
\begin{theorem}\label{lbdensity} For $x \geq 2$, we have $$D'(x) \gg \frac{x}{\log x}.$$ \end{theorem}

\begin{proof} We will show that a positive proportion of $2$-dense integers are strictly $2$-dense, except for some possible obstructions at small primes. We will then use a counting argument to deal with these small obstructions. We begin by counting integers $n \leq x$ that are $2$-dense but not strictly $2$-dense. Let $n = mpj,$ where $m$ is a $2$-dense integer, $p$ is a prime satisfying $m < p < 2m$, and $j$ is an integer that has the following properties: $j \leq x/mp$, $P^-(j) > p$ and the prime factors of $j$ satisfy the conditions necessary for $mpj$ to be $2$-dense. Since we have specified that $m < p < 2m$ then, by Lemma \ref{TCondition}, $mpj$ is $2$-dense. However, we know from Lemma \ref{T'Condition} that $mpj$ will not be strictly $2$-dense for values of $p$ in this range. Thus, the integers $n$ that we have constructed are all $2$-dense but not strictly $2$-dense. By varying the sizes of $m$ and $j$, as well as our choice of the prime $p$, we can construct, in this manner, all integers that are $2$-dense but not strictly $2$-dense.

Now, let $C > 16$ be an integer that is chosen to be large relative to the size of the constant $\kappa_1$ from Lemma \ref{SaiasD}. For each integer $k > C$, consider those $2$-dense numbers $m \in (2^{k-1}, 2^k]$. Since $m < p < 2m$, we must have $p \in (2^{k-1}, 2^{k+1}).$ We will say that $n$ has an \textit{obstruction at $k$} if $m$ and $p$ land within these intervals, i.e., if $p$ is a prime in our construction that prevents $n$ from being strictly $2$-dense. Thus, for values of $k$ in this range, the number of $2$-dense integers that are not strictly $2$-dense is at most \begin{equation}\label{sums} \sum_{k > C} \ \ \sum_{\substack{m \in (2^{k-1}, 2^k) \\ m \ \twod}} \ \ \sum_{\substack{p \in (2^{k-1}, 2^{k+1}) \\ p \ \mathrm{prime}}} \ \sum_{\substack{j \leq x/mp \\ mpj \ \twod \\ P^-(j) > p}} 1.\end{equation} 

Observe that, when $j = 1$, we are counting \begin{equation}\label{j1}\#\{n \leq x: n = mp, m < p < 2m, m \ \hbox{is} \ 2\hbox{-dense}\}.\end{equation} The conditions that $m \leq x/p$ and $m < p$ together imply that $m < x/m$, i.e., $m < \sqrt{x}$. Similarly, the conditions on $p$ force us to have $p < \sqrt{2x}$. Thus, the quantity counted in \eqref{j1} is at most \begin{equation}\label{otherjsum} 1 + \sum_{p < \sqrt{2x}} \#\{m \leq \sqrt{x}: m \ \hbox{is} \ \twod\}.\end{equation} From Lemma \ref{SaiasD}, we have \begin{equation}\label{jsummation} \#\{m \leq \sqrt{x}: m \ \hbox{is} \ \twod\} = D(\sqrt{x}) \ll \frac{\sqrt{x}}{\log \sqrt{x}} \ll \frac{\sqrt{x}}{\log x}.\end{equation} Moreover, the number of terms in the summation is at most $\pi(\sqrt{2x}).$ By Chebyshev's Inequality (cf. \cite[Theorem 3.5]{pollack}), $$\pi(\sqrt{2x}) \ll \frac{\sqrt{2x}}{\log \sqrt{2x}} \ll \frac{\sqrt{x}}{\log x}.$$ Therefore, the quantity counted in \eqref{otherjsum} is bounded above by a constant times $$\frac{\sqrt{x}}{\log x} \cdot \frac{\sqrt{x}}{\log x} = \frac{x}{\log^2 x},$$ which is negligible compared with the magnitude of $D(x)$ given in Lemma \ref{SaiasD}. 

Assume hereafter that $j > 1$ and, for now, let $k$ be fixed. Our first order of business will be to bound the integers counted by the sum in $j$ in \eqref{sums}. We have \begin{align}\sum_{\substack{1 < j \leq x/mp \\ mpj \ \twod \\ P^-(j) > p}} 1 & = \#\{n \leq x : mp \mid n, n \ \hbox{is} \ \twod\} \notag \\ & = \sum_{q \leq x} \# \{qM \leq x : mp \mid M, M \ \hbox{is} \ 2\hbox{-dense}, q \ \hbox{is prime}, P(M) < q < 2M\} \label{newBuchstab} \\ & \leq \sum_{q \leq x/mp} D(x/mpq, q) \notag,\end{align} where the second line follows from Lemma \ref{TCondition} and the third line follows from the definition of $D(x, y).$ We can improve upon this final bound slightly. Namely, since we are counting integers $M$ with $\frac{q}{2} < M \leq \frac{x}{q}$ in the second line, we see that $D(x/mpq, q) = 0$ when $q > \sqrt{2x}$. Now, to simplify our notation, let $z = x/mp$. Then, using the bound that we just derived for $q$ in conjunction with \eqref{newBuchstab} yields \begin{align*}  \sum_{\substack{j \leq z \\ mpj \ \twod \\ P^-(j) > p}} 1 & \leq \sum_{q \leq \mathrm{min}(z, \sqrt{2x})} D(z/q, q) \\ & \leq c \sum_{q \leq \mathrm{min}(z, \sqrt{2x})} H(z/q, q), \end{align*} where the second inequality follows from Lemma \ref{newsaias}. Since $j > 1$ and $P^-(j) > p$, then $j > p > 2^{C}$ implies, in particular, that $z > 2^{16}$. Thus, we can apply Lemma \ref{saiasmain} to obtain an upper bound of $cH(z, z),$ since $\ell(z) \geq 1$ when $z \geq 2^{16}.$ By appealing to the definition of the function $H(x, y)$, we have \begin{equation}\label{mpeq} \sum_{\substack{j \leq z \\ mpj \ \twod \\ P^-(j) > p}} 1 \ll H(z, z) \ll \frac{z}{\log z} = \frac{x}{mp \log (x/mp)}.\end{equation} 

Now, it will suffice to replace $\log(x/mp)$ with $\log x$ in the denominator of \eqref{mpeq}. The argument boils down to showing that $m < \exp\{C \sqrt{\log x}\}$, which we will now demonstrate. Instead of estimating the full sum in $j$, we could ignore the stipulation that $mpj$ is $2$-dense and use a cruder estimate for $$\sum_{\substack{j \leq x/mp \\ mpj \ \mathrm{squarefree} \\ P^-(j) > p}} 1.$$ Since $x/mp > p$, we can use Brun's Sieve (cf. \cite[Theorem 2.2]{hr}) to show that $$\#\{j \leq x/mp : q \mid j, q \ \hbox{prime} \ \Rightarrow q > p\} \ll \frac{x}{mp} \prod_{q \leq p} \left( 1 - \frac{1}{q}\right).$$ Moreover, since $2^{k-1} < p \leq x$, we can apply Mertens' Theorem (cf. \cite[Theorem 3.15]{pollack}), which allows us to obtain $$\prod_{q \leq 2^{k-1}} \left(1 - \frac{1}{q}\right) \ll \frac{1}{\log 2^{k-1}} \ll \frac{1}{k}.$$ Thus, we have the following crude estimate for the sum in $j$: $$\frac{x}{mp} \prod_{q \leq 2^{k-1}}\left(1 - \frac{1}{q}\right) \ll \frac{x}{mpk}.$$ Using our crude estimate in \eqref{sums} yields $$\sum_{\substack{m \in (2^{k-1}, 2^k) \\ m \ \twod}} \ \sum_{\substack{p \in (2^{k-1}, 2^{k+1}) \\ p \ \mathrm{prime}}} \ \sum_{\substack{j \leq x/mp \\ mpj \ \mathrm{squarefree} \\ P^-(j) > p}} 1 \ll \sum_{\substack{m \in (2^{k-1}, 2^k) \\ m \ \twod}} \frac{x}{mk} \sum_{\substack{p \in (2^{k-1}, 2^{k+1}) \\ p \ \mathrm{prime}}} \frac{1}{p} .$$  Now, the largest term in the sum in $p$ is at most $\frac{1}{2^{k-1}}$, and the number of terms in this sum is certainly less than $\pi(2^{k+1})$. Hence, by Chebyshev's Inequality (cf. \cite[Theorem 3.5]{pollack}), we have $$\sum_{\substack{p \in (2^{k-1}, 2^{k+1}) \\ p \ \mathrm{prime}}} \frac{1}{p} \ll \frac{1}{2^{k-1}} \cdot \frac{2^{k+1}}{\log 2^{k+1}} \ll \frac{1}{k}.$$ As a result, $$\sum_{\substack{m \in (2^{k-1}, 2^k) \\ m \ \twod}} \frac{x}{mk} \sum_{\substack{p \in (2^{k-1}, 2^{k+1}) \\ p \ \mathrm{prime}}} \frac{1}{p}  \ll \frac{x}{k^2} \sum_{\substack{m \in (2^{k-1}, 2^k) \\ m \ \twod}} \frac{1}{m}.$$ Similarly, we can use the fact that the largest term in the sum in $m$ is at most $\frac{1}{2^{k-1}}$ and the number of terms is less than $D(2^k).$ Thus, using the upper bound for $D(x)$ given in Lemma \ref{SaiasD}, we obtain $$\frac{x}{k^2} \sum_{\substack{m \in (2^{k-1}, 2^k) \\ m \ \twod}} \frac{1}{m} \ll \frac{x}{k^3}.$$ Summing over all values of $k > y$ allows us to see that $$\sum_{k > y} \frac{x}{k^3} \leq x \int_{y-1}^\infty \frac{1}{t^3} dt = \frac{x}{2(y-1)^2}.$$ In particular, when $y \geq 1 + C\sqrt{\log x}$, we have $$\sum_{k > y} \frac{x}{k^3} \leq \frac{x}{C^2 \log x}.$$ Thus, if $C$ is large, then the number of $2$-dense integers with obstructions at $k \geq C \sqrt{\log x}$ is small relative to the number of $2$-dense integers. Hence, it suffices to take $k < C \sqrt{\log x}$ in our computations, which means that we can take $m < \exp\{C \sqrt{\log x}\}.$ Now, since $p < 2m$ in our construction, then $pm < 2m^2 < 2 \exp\{2C \sqrt{\log x}\}$. Therefore, $$\frac{x}{mp \log(x/mp)} \leq \frac{x}{mp \log(x(2 \exp\{2C \sqrt{\log x}\})^{-1})} \ll \frac{x}{mp \log x}.$$ 

As we have just shown, we can use $\log x$ in lieu of $\log(x/mp)$ in \eqref{mpeq} in order to arrive at a more precise estimate for the sum in $j$. In other words, we have $$\sum_{\substack{m \in (2^{k-1}, 2^k) \\ m \ \twod}} \sum_{\substack{p \in (2^{k-1}, 2^{k+1}) \\ p \ \mathrm{prime}}} \sum_{\substack{j \leq x/mp \\ mpj \ \twod \\ P^-(j) > p}} 1 \ll \sum_{\substack{m \in (2^{k-1}, 2^k) \\ m \ \twod}} \frac{x}{m \log x} \sum_{\substack{p \in (2^{k-1}, 2^{k+1}) \\ p \ \mathrm{prime}}} \frac{1}{p} .$$ Moreover, using the same estimates for the summations in $m$ and $p$ that we found above allows us to obtain $$\sum_{\substack{m \in (2^{k-1}, 2^k) \\ m \ \twod}} \frac{x}{m \log x} \sum_{\substack{p \in (2^{k-1}, 2^{k+1}) \\ p \ \mathrm{prime}}} \frac{1}{p} \ll \frac{x}{k^2 \log x}.$$ Summing over all values of $k > C$ allows us to see that $$\frac{x}{\log x} \sum_{k > C} \frac{1}{k^2} \leq \frac{x}{\log x} \int_C^\infty \frac{1}{t^2} dt \leq \frac{x}{C \log x}.$$ Since we have chosen $C$ to be large relative to the size of Saias' lower bound constant for $D(x)$, then our count of $2$-dense integers with obstructions at $k > C$ is negligible relative to the full count of $2$-dense integers. 

All that remains, then, is for us to consider the $2$-dense integers $n$ that have no obstructions at values of $k > C$. Let $$\mathcal{N} = \{n \leq x: n \ \hbox{is} \ 2\hbox{-dense} \ \hbox{with no obstructions at} \ k > C\}.$$ Since we chose $C$ to be large relative to the implicit constant $\kappa_1$ from Lemma \ref{SaiasD}, we can use this lemma, along with our count of $2$-dense integers with obstructions at $k > C$, to show that for all large $x$, $$\# \mathcal{N} \geq \kappa \frac{x}{\log x},$$ where $\kappa > 0$ is some absolute constant. Define $f: \mathcal{N} \rightarrow \Z^+$ to be a function that maps each element $n \in \mathcal{N}$ to its largest $2$-dense divisor $m$ with all prime factors at most $2^C$. Let $\mathcal{M} = \mathrm{Im} f.$  By the Pigeonhole Principle, there is some $m_0 \in \mathcal{M}$ with at least $\frac{\# \mathcal{N}}{\# \mathcal{M}} \geq \frac{\kappa}{4^{2^C}} \frac{x}{\log x}$ elements in its pre-image,  since Chebyshev's bound (cf. \cite[pg. 108]{pollack}) implies $\prod_{p \leq 2^C} p \leq 4^{2^C}.$ In other words, $m_0$ divides at least the average number of integers in a pre-image. For each $n \in \mathcal{N}$ with $f^{-1}(m_0) = n$, let $$n' = n \prod_{\substack{p \leq 2^C \\ p \ \mathrm{prime} \\ p \nmid m_0}} p.$$ Then, $n'$ is squarefree, since the only primes dividing $n$ that are smaller than $2^C$ are also divisors of $m_0$. Moreover, $n'$ is strictly $2$-dense, since the strict inequality form of Bertrand's Postulate implies that the product of all primes up to $2^C$ is strictly $2$-dense. Finally, since we multiplied every $n$ in the pre-image of $m_0$ by the same sequence of primes, there is a one-to-one correspondence between the strictly $2$-dense integers up to $4^{2^C}x$ that we have constructed and the $2$-dense numbers in the pre-image of $m_0$. Thus, at least $\frac{\kappa}{4^{2^C}} \frac{x}{\log x}$ of the integers up to $4^{2^C}x$ are strictly $2$-dense. Therefore, since $D'(4^{2^C}x) \geq \frac{\kappa}{4^{2^C}} \frac{x}{\log x}$, we have the stated result.\end{proof}

The proof of Theorem \ref{lbdensity} can be used to obtain the following results on the relationship between the practical and $\varphi$-practical numbers. 

\begin{corollary}\label{corollary1} For $x$ sufficiently large, we have  $$\#\{n \leq x: n \ \hbox{is practical but not} \ \varphi\hbox{-practical}\} \gg \frac{x}{\log x}.$$ \end{corollary}

\begin{proof} As in the proof of Theorem \ref{lbdensity}, let $$\mathcal{N} = \#\{n \leq x: n \ \hbox{is} \ 2\hbox{-dense with no obstructions at} \ k > C\}.$$ We know that, for sufficiently large $x$, we have $\# \mathcal{N} \geq \kappa \frac{x}{\log x}$, where $\kappa > 0$ is some absolute constant. As before, let $f: \mathcal{N} \rightarrow \Z^+$ map each $n \in \mathcal{N}$ to its largest $2$-dense divisor $m$ satisfying the condition that $P(m) \leq 2^C.$ Then, from the proof of Theorem \ref{lbdensity}, there exists some $m_0 \in \hbox{Im} f$ with at least $\frac{\kappa}{4^{2^C}} \frac{x}{\log x}$ elements in its pre-image. For each $n \in \mathcal{N}$ with $f^{-1}(m_0) = n$, let $$n' = \frac{2 \cdot 7^2}{\gcd(15, m_0)} n \prod_{\substack{7 < p \leq 2^C \\ p \nmid m_0}} p.$$ Since $n$ is squarefree and $2$-dense, then $2^2 \| n'.$ Thus, we can write $n' = 28M$, where $M$ is an integer with $P^-(M) \geq 7.$ Observe that $28$ is not $\varphi$-practical, since $t^{28} - 1$ has no divisor with degree $5$. Moreover, as all prime divisors of $M$ are at least $7$, it follows that $t^{n'}-1$ has no divisor of degree $5$. Hence, $n'$ is not $\varphi$-practical. On the other hand, let $$l = n \prod_{\substack{7 < p \leq 2^C \\ p \nmid m_0}} p.$$ Since $n$ is $2$-dense, Bertrand's Postulate implies that $l$ is $2$-dense, hence practical. In particular, this means that all of the prime factors of $l$ must satisfy the inequality from Proposition 1.3. Now, let $r = \frac{2 \cdot 7^2}{\gcd(15, m_0)}.$ Since $2 \cdot 7^2 > 15$ then, as each prime $p$ dividing $l$ satisfies $p \leq \sigma(m) + 1$, it must be the case that $p \leq \sigma(rm) + 1.$ Therefore, $n'$ is practical. 

In order to construct the integers $n'$, we multiplied every $n$ in the pre-image of $m_0$ by the same number. As a result, there is a one-to-one correspondence between the practical numbers up to $r \cdot 4^{2^C} x$ that we have constructed and the $2$-dense numbers in the pre-image of $m_0.$ Therefore, at least $\frac{\kappa}{4^{2^C}} \frac{x}{\log x}$ of the integers up to $r \cdot 4^{2^C}$ are practical but not $\varphi$-practical. \end{proof}

\begin{corollary}\label{phinotpr} For $x$ sufficiently large, we have $$\#\{n \leq x: n \ \hbox{is $\varphi$-practical but not practical}\} \gg \frac{x}{\log x}.$$\end{corollary}

\begin{proof} In Theorem \ref{lbdensity}, we showed that $\# \{n \leq x: n \ \hbox{even, squarefree and} \ \varphi\hbox{-practical}\} \gg \frac{x}{\log x}.$ Now, either a positive proposition of the integers counted in this set are divisible by $7$ or a positive proportion are not divisible by $7$. In the first case, let $n$ be a strictly $2$-dense number that is divisible by $7$, and let $n' = \frac{3}{2} n.$ In the second case, we can take $n$ to be a strictly $2$-dense number with $(7, n) = 1$, and $n' = \frac{21}{2} n.$ In either case, $n'$ can be re-written in the form $$n' = 3^2 \cdot 5 \cdot 7 \cdot \prod_{\substack{7 < p \leq x \\ p \mid n}} p.$$ Since $3^2 \cdot 5 \cdot 7$ is $\varphi$-practical and $n$ is strictly $2$-dense, then $n'$ is $\varphi$-practical by Lemmas \ref{keylemma} and \ref{modified}. However, $n'$ is not practical since it is odd. \end{proof}

We remark that Corollary \ref{phinotpr} also shows that a positive proportion of $\varphi$-practical numbers are odd.

\textit{Acknowledgements.} I would like to thank my adviser, Carl Pomerance, for sparking my interest in the $\varphi$-practical numbers and for giving me a number of suggestions that allowed me to simplify the contents of this paper substantially. As always, I am grateful for the time and care that he devotes to reviewing my work. I would also like to thank the anonymous referee for suggesting several improvements on the clarity of the exposition.

\providecommand{\bysame}{\leavevmode\hbox
to3em{\hrulefill}\thinspace}
\providecommand{\MR}{\relax\ifhmode\unskip\space\fi MR }
% \MRhref is called by the amsart/book/proc definition of \MR.
\providecommand{\MRhref}[2]{%
  \href{http://www.ams.org/mathscinet-getitem?mr=#1}{#2}
} \providecommand{\href}[2]{#2}

\end{document}